 \documentclass[final,1p,times, authoryear]{elsarticle}
\usepackage{amssymb}
\usepackage{amsthm}

\usepackage[unicode,pdfborder={0 0 0}]{hyperref}
\usepackage{amsmath}
\usepackage{dsfont}
\usepackage{amsfonts}
\usepackage{xcolor}
\usepackage{soul}
\usepackage{graphicx}
\usepackage{subcaption}
\usepackage{mathrsfs}
\usepackage{amsmath, amsthm}
\usepackage{amsfonts,amssymb,dsfont}
\usepackage{nicefrac,mathrsfs}
\usepackage{bm}                                 
\usepackage{algorithm}
\usepackage{algorithmicx}
\usepackage{algpseudocode}
\usepackage{graphicx}
\usepackage{subcaption}
\usepackage[backgroundcolor=white,bordercolor=orange]{todonotes}

\usepackage{enumerate}

\newcommand{\ccA}{{\mathscr A}}  
\newcommand{\ccL}{{\mathscr L}}
\newcommand{\ccB}{{\mathscr B}}

\newcommand{\ccF}{{\mathscr F}}

\newcommand{\ccP}{{\mathscr P}}
\newcommand{\ccE}{{\mathscr E}}

\newcommand{\ccN}{{\mathscr N}}

\renewcommand{\ccP}{{\mathscr P}}

\newcommand{\Ind}{{\mathds 1}}
\newcommand{\ind}[1]{\Ind_{\{#1\}}}

\newcommand{\R}{\mathbb{R}}

\newcommand{\bbF}{\mathbb{F}}

\newcommand{\N}{\mathbb{N}}

\usepackage[utf8]{inputenc}

\numberwithin{equation}{section}

\newtheorem{assumption}{Assumption}[section]

\newtheorem{theorem}[assumption]{Theorem}

\newtheorem{proposition}[assumption]{Proposition}

\newtheorem{lemma}[assumption]{Lemma}

\theoremstyle{definition}

\newtheorem{remark}[assumption]{Remark}
\newtheorem{example}[assumption]{Example}

\begin{document}

\begin{frontmatter}

\title{Nonlinear filtering with stochastic discontinuities}

 \author[1]{Thorsten Schmidt\corref{cor}}
 \ead{thorsten.schmidt@stochastik.uni-freiburg.de}

 \author[1]{Félix B. Tambe-Ndonfack}
 \ead{felix.ndonfack@stochastik.uni-freiburg.de}
 
\affiliation[1]{organization={Department of Mathematical Stochastics, Mathematical Institute, University of Freiburg },
  addressline={Ernst-Zermelo-Str. 1},
  postcode={79104},
  city={Freiburg},
  country={Germany}} 
    
\begin{abstract}
Filtering problems with jumps in both the signal and the observation have been extensively studied, typically under the assumption that jump times are totally inaccessible. In many applications, however, jump times are known in advance (i.e., predictable), such as scheduled clinical visits, dividend payment dates, or inspection times in engineering systems.

Taking predictable jump times as a starting point, we investigate a filtering problem in which both the signal and the observations can exhibit jump at predictable times.
We derive the corresponding Kushner–Stratonovich and Zakai equations, thereby extending classical nonlinear filtering results to a setting with predictable discontinuities.

We illustrate the framework on a Kalman filtering model with predictable jumps and on applications to longitudinal clinical studies, such as spinal muscular atrophy (SMA), as well as to machine learning models (neural jump ODEs) and credit risk.
\end{abstract}


\end{frontmatter}

\section{Introduction}\label{sec:Introduction}

Filtering theory is concerned with estimating the distribution of an unobserved (hidden) signal process from partial and noisy observations. Classical continuous-time results, including  the Kushner-Stratonovich and Zakai equations, provide recursive stochastic equations for the conditional law and form a foundation for applications ranging from engineering to mathematical finance and statistics (see, e.g., \cite{bain2009crisan}).

A substantial part of the filtering literature deals with systems that exhibit jumps. Already early semimartingale approaches established the structural form of nonlinear filtering equations  in a high degree of generality, see \cite{grigelionis1972stochastic,grigelionis2006stochastic}. Subsequent  developments have predominantly focused on the case where jumps are \emph{totally inaccessible}, for instance when driven by L\texorpdfstring{\'e}{e}vy processes or Poisson random measures. In this setting, the canonical form of the compensator and established  martingale representation theorems greatly simplify derivations of filtering equations.

In numerous applied contexts, however, information arrives at \emph{predictable} (i.e. scheduled) times, inducing jumps in the observable filtration. Such discontinuities, which occur on \emph{thin predictable sets}, have recently been formalized under the concept of \emph{stochastic discontinuities}.

This framework provides a systematic way to model processes with jumps at predictable times and has led to recent theoretical developments, for instance in affine processes, see  \cite{keller2019affine}.
The importance of predictable jump times has also been widely recognized in the financial and econometrics literature; see, e.g.,
\cite{Fama1970,merton1974pricing,GeskeJohnson84,Piazzesi2005,duffie2001term,BelangerShreveWong2004,GehmlichSchmidt2015MF,fontana2020term, fontana2024term}.
The corresponding semimartingale calculus required in such settings is essential for modelling events such as discrete dividend payments, scheduled clinical visits, planned sensor readings, or event-driven updates in machine-learning models. 

Several important contributions in nonlinear filtering have addressed models with jumps. The work of \cite{ceci2012nonlinear,ceci2014zakai} provides a comprehensive treatment of correlated jump-diffusion signals and observations, deriving and analysing the corresponding Kushner-Stratonovich and Zakai equations.  In particular, they clarify how a common jump structure modifies the usual filtering identities and how one can recover a linear Zakai under a suitable change of measure. 
In mathematical finance, \cite{frey2010pricing,frey2012pricing}  applied filtering with jumps to credit risk, where default events generate discontinuous information flows. More recently, \cite{bandini2022stochastic} study filtering for pure-jump signals, allowing both the signal and the observation to jump at predictable times and highlighting the distinct mathematical features that arise compared to the classical totally inaccessible case.

Despite these advances, a systematic treatment of  filtering equations for a \emph{diffusion with scheduled jumps} signal, observed  at predictable times in a noisy environment, is still lacking. This setting presents a key technical challenge: the observation filtration itself possesses predictable jump times, and the signal may jump systematically  at these scheduled instants. A complete analysis requires combining tools from the  theory of stochastic discontinuities with the classical change of measure and martingale methods from nonlinear filtering.

This paper provides such an analysis. We introduce a semimartingale model in which the signal evolves as a diffusion between a finite sequence of predictable times and may jump at these times. The observation process is a pure-jump process that updates only at these times, with each increment consisting of a signal-dependent term plus additive noise. Within this framework, we obtain the following main contributions:
\begin{enumerate}[1.]
    \item \emph{Explicit filtering equations with predictable jumps.} We derive the Kushner-Stratonovich for the conditional distribution of the signal under the observation filtration. The resulting equation contains an additional predictable-jump term and an explicit expression for the compensated random measure integrand, given by a Radon-Nikodym derivative involving  predictable projections of signal dependent jump kernels. This representation clarifies which parts of the classical structure persists and how predictable jump modify the innovation term.
    \item \emph{Zakai formulation via a tailored change of measure}. By constructing a Girsanov-type transformation adapted to the random measure setting, we obtain an unnormalised filter satisfying a linear Zakai equation. This formulation is analytically and computationally advantageous, extending the classical change of measure approach to filtrations with predictable jump times. Our approach builds on standard results for changes of measure for integer-valued random measures and martingale problems for semimartingales.
\end{enumerate}
To illustrate the applicability of our results, we discuss examples including an Ornstein–Uhlenbeck-type signal with scheduled jumps, which leads to a Kalman-like update at deterministic times, and a stylized health monitoring model in which clinical visits occur at predictable times. We further outline applications to machine learning models with event-driven updates (neural jump ODEs) and to credit risk. 

The remainder of the paper is organized as follows: Section \ref{sec 2:Filtering theory} introduces the semimartingale and filtering framework, Section \ref{sec3: motivating examples} presents our examples, Section \ref{sec4: proof of KS} derives the Kushner-Stratonovich equation, and Section \ref{sec5: zakai formulation} develops the Zakai formulation via a change of measure.

\section{The filtering framework}\label{sec 2:Filtering theory}
In this section, we introduce the filtering framework and specify the underlying stochastic model. We begin by describing the signal and observation processes within a semimartingale setting; for the theoretical background we refer to \cite{jacod2013limit}. 

Consider a filtered probability space $(\Omega,\ccF,\bbF,P)$ satisfying the usual conditions, i.e.\ $\bbF=(\ccF_t)_{t \ge 0}$ is right-continuous and all subsets of sets with probability zero are contained in $\ccF_0$. 
By $\ccN$ we denote the nullsets of $P$.
The predictable $\sigma$-field $\ccP$ is generated by all processes which are left-continuous (c\`ag) while the optional $\sigma$-field is generated by all c\`adl\`ag processes. 
A random time $T$ is a measurable mapping $T:\Omega \to \R_{\ge 0} \cup \{\infty\}$. It is called a \emph{stopping time}, if $\{T \le t\} \in \ccF_t$ for all $t \ge 0$. It is called a \emph{predictable} time if it is \emph{announced} by stopping times $T_n \uparrow T$ with $T_n < T$.

The filtering problem consists of an unobserved \emph{signal} $X$ and an \emph{observation} $Y$. The driving processes are an $m$-dimensional Brownian motion $B$ and a pure-jump process $J$ whose jumps arise at times which are predictable in the observation filtration.   More precisely, consider a fixed finite number $K\in \N$ of jump times, and $m+n$-dimensional random variables $Z_1,\dots,Z_K$, which are i.i.d.\ and  independent of $B$. We denote in the following $Z_i=(\xi_i,\eta_i)$ where $\xi_i \in \R^m$ and $\eta_i \in \R^n$. 
Jumps arise at the random times (to be made more precise later) $T_1,\dots,T_K$ and the jump process $J$  is given by
$$ J_t = \sum_{i=1}^K \ind{T_i \le t} \xi_i, \qquad t \ge 0. $$
The signal is the special semimartingale given by the unique strong solution of the stochastic differential equation (SDE)
\begin{align} \label{eq:X}
        dX_t &= a(X_{t})\, dt + b(X_{t}) \, dB_t+ c(X_{t-}) \, dJ_t,\quad X_0=x_0\in \R^m \qquad t \ge 0,
\end{align} 
with measurable functions 
\begin{align}\label{assumption 0}
    a: \R^m \to \R^m,\quad b,c: \R^m \to \R^{m\times m}.
\end{align}
The observation occurs at predictable times,  predictable in the observation filtration, and is given by a function of the signal plus additive noise:  we assume that 
\begin{align}\label{eq:Y}
	Y_t &= \sum_{i=1}^K \ind{T_i \le t}\,  \Big( f(X_{T_i-},Y_{T_i-}) + \eta_i \Big), \quad t \ge 0
\end{align}
where $f: \R^m \times \R^n  \to \R^n$ is an observation function, which is measurable.
Additionally, we assume 
that $(T_i)_{i=1}^K$ are positive and predictable with respect to the (augmented) filtration generated by $Y$, which we denote by $\bbF^Y=(\ccF_t^Y)_{t \ge 0}$ where $\ccF_t^Y = \sigma(Y_s \colon s \le t) \vee \ccN$. Note that this filtration is right-continuous by construction.

For simplicity, the observation filtration is assumed to be piecewise constant. 
It is possible to extend this setting by including additional, continuously observed covariates here, however, at the cost of complicating the results significantly.

\subsection{Representation through random measures}
For the following, it will be useful to utilise random measures to describe the jumps, and in particular their compensators. 
Before introducing random measures, we recall the concept of predictable projections for processes, see again \cite{jacod2013limit} for details.

For a $d$-dimensional $\ccF \otimes \ccB(\R_{\ge 0})$-measurable stochastic process, there always exists a unique process ${}^p X$, called the \emph{predictable projection}, which is itself predictable and satisfies $({}^pX)_T=E[X_T|\ccF_{T-}]$ on $T< \infty$ for all predictable times $T$.

Next, a random measure on $\R_{\ge 0} \times \R^n$ is a family $\mu=(\mu(\omega;dt,dy) \colon \omega \in \Omega)$ of nonnegative measures such that $\mu(\omega;\{0\}\times \R^n)= 0$ identically.
We introduce $\tilde \Omega = \Omega \times \R_{\ge 0} \times \R^n$ and $\tilde \ccP=\ccP \otimes \ccB(\R^n)$ and call a function $W:\tilde \Omega \to \R^n$ predictable if it is $\tilde \ccP$-measurable. For those, we define the stochastic process $W*\mu$ through
\begin{align}\label{stoch process}
    (W*\mu)_t(\omega) = \int_{[0,t]\times \R^n} W(\omega; s,y) \, \mu(\omega;ds,dy),
\end{align}
if the integral w.r.t.\ $|W|$ is finite, and $+\infty$ otherwise. The random measure itself is called predictable if $W*\mu$ is predictable for all predictable functions $W$.
For a random measure $\mu$ there exists a unique predictable measure $\nu$, called \emph{compensator} or dual predictable projection, which satisfies $E[(W*\mu)_\infty]=E[(W*\nu)_\infty]$ for every nonnegative predictable $W$.

For us, the most relevant  random measure, which we will denote by $\mu$ in the following, is the one  associated to the observation $Y$ and is given by
\begin{align} \label{eq:mu}
	\mu(dt,dy) = \sum_{i=1}^K \delta_{(T_i,\Delta Y_{T_i})} (dt,dy),
\end{align} 
where, according to Equation 
\eqref{eq:Y}, $\Delta Y_{T_i}= f(X_{T_i-},Y_{T_i-})+\eta_i,$ $i=1,\dots,K$.

We denote the  distribution  of  $\eta_1$ by $F:=P(\eta_1 \in \cdot)$ and introduce
the regular conditional distribution of $\Delta Y_{T_i}$ given $\ccF_{T_i-}^Y$ by
\begin{align}\label{conditional distribution of jump size:2}
    F^i(A) := E\big[F(A-f(X_{T_i-},Y_{T_i-})) | \ccF_{T_i-}^Y\big],\quad A \in \ccB(\R^n)
\end{align}

where 
$
A-x =\{y-x \colon y \in A\}$.
Note that $F^1,\dots,F^K$ are only defined up to null sets, but since we are only considering finitely many observations this does not pose a problem.

\begin{lemma}\label{Compensated random measure}
    The $\bbF^Y$-compensator\footnote{Also known in the literature as the dual predictable projection, \citep{jacod2013limit}} of $\mu$ is given by the predictable random measure
    \begin{align}\label{equ2: expression of mu^p}
        \nu(dt, dy) = \sum_{i=1}^K \delta_{T_i}(dt) F^i(dy).
    \end{align}
\end{lemma}
\begin{proof}
We note that $\nu$ is predictable by construction. Moreover, for a simple nonnegative and predictable $W$, i.e.\ when
$$ W(\omega;s,y) = \Lambda(\omega) \Ind_{(t,t']}(s) \Ind_A(y)$$
with an $\ccF^Y_t$-measurable random variable $\Lambda$ and $A\in \ccB(\R^n)$ we obtain that
\begin{align*}
    E[(W*\mu)_\infty] &= \sum_{i=1}^K E\Big[ \,E\big[\Lambda\ind{T_i \in (t,t']} \ind{\Delta Y_{T_i} \in A} |\ccF_{T_i-}^Y\big] \Big] \\
    &= \sum_{i=1}^K E\Big[\Lambda\ind{T_i \in (t,t']} \,E\big[ \ind{\Delta Y_{T_i} \in A} |\ccF_{T_i-}^Y\big] \Big]. 
\end{align*}
Since $\Delta Y_{T_i}= f(X_{T_i-},Y_{T_i-})+\eta_i$ we obtain that
\begin{align*}
    E\big[ \ind{\Delta Y_{T_i} \in A} |\ccF_{T_i-}^Y\big]
    &= P(\eta_i \in A-f(X_{T_i-},Y_{T_i-}) |\ccF_{T_i-}^Y)
    = F^i(A).
\end{align*}
The result now follows by linearity and an application of the monotone class theorem.
\end{proof}

\subsection{The model under full information} 
The first step is to construct the model and the associated full information filtration $\bbF$.
To this end, we introduce the random measure
$$ 
\psi(dt,dz) = \sum_{i=1}^K \delta_{(T_i, Z_i)}(dt,dz). 
$$
Since $Z_i=(\xi_i,\eta_i)$, we obtain with Equation \eqref{eq:Y} that
$\Delta Y_{T_i} = f(X_{T_i-},Y_{T_i-})+\eta_i$, such that
 \begin{align}\label{Relation btw two measures}
     (W* \mu)_t = 
     \int_0^t\int_{\R^{m + n}} W(s,f(X_{s-},Y_{s-}) + \eta)\,\psi (ds,d\xi \otimes d\eta), \qquad t \ge 0.
 \end{align}
The signal $X$ and the observation $Y$ can now be represented in the following form,
\begin{align}\label{sys: 2}
    \begin{aligned}
        dX_t &= a(X_t) \, dt + b(X_t) \, dB_t + \int_{\R^{m + n}} c(X_{t-})\, \xi \, \psi(dt,d\xi \otimes d\eta), \\
        dY_t &= \int_{\R^{m + n}} \big( f(X_{t-},Y_{t-}) + \eta) \, \psi(dt,d\xi \otimes d\eta).
    \end{aligned}
\end{align}
For the following, we assume that 
 $\bbF$ is given by $\ccF_t = \sigma(X_s,Y_s \colon s \le t) \vee \ccN$.
As in Lemma \ref{Compensated random measure}, we obtain that the $\bbF$-compensator of $\psi$ is given by the predictable measure $\sum_{i=1}^K \delta_{T_i}(dt)F_Z(dz)$, where $F_Z$ denotes the distribution of $Z_1$. Denote by $P(\xi\in d\xi|\eta = \eta_0)=F_{\xi|\eta}(d\xi|\eta_0)$  the conditional distribution of $\xi$  given $\eta = \eta_0 $ under the joint law $F_Z$.

\subsection{The Kushner-Stratonovich equation}
Filtering is the dynamic, recursive representation of the conditional distribution of the signal. In the semimartingale-framework we consider here, this conditional distribution can be equivalently characterized by the family of stochastic processes $\pi(\phi)$ where $\phi$ range over all twice continuously differentiable and bounded functions and 
\begin{align} \label{eq:pi_t}
    \pi_t(\phi) = E[\phi(X_t) | \ccF_t^Y ], \qquad t \ge 0.
\end{align}
However, we need to choose an appropriate regularisation of the process $\pi$. Theorem $2.24$ in \citet{bain2009crisan} shows that $\pi$ can be chosen optional and in such a way that  \eqref{eq:pi_t} holds almost surely.

For $\phi \in \mathcal{C}_b^{2}(\R^m; \R)$, i.e.\ a continuous function $\phi: \R^m \to \R$ such  its first and second derivatives  are bounded and continuous,
denote by
$\ccL$ the generator of the continuous part of $X$, i.e.
\begin{align}\label{eq:ccL}
\ccL\phi (x) :=  \sum_{i=1}^m a_i(x) \dfrac{\partial \phi}{\partial x_i}(x) + \frac{1}{2} \sum_{i,j=1}^m \sigma_{ij}(x) \frac{\partial^2 \phi}{\partial x_i \partial x_j}(x),
\end{align}
where $\sigma := b(x)b(x)^\top$ and furthermore, for the jump part,
\begin{align}\label{eq: ccA}
    \ccA \phi(x):=\int_{\R^{m+n}} \big(\phi(x+c(x)\xi) - \phi(x)\big)F_Z(d\xi \otimes d\eta).
\end{align}

\begin{assumption} \label{assumption 1}
Assume that for all $t \ge 0$, the functions $a,b,c,f$ from Equation \eqref{eq:X} satisfy
\begin{align*}
        E\bigg[\displaystyle\int_0^t \int_{\R^n} \nu(ds,dy) \bigg]< \infty
        ,\quad 
        E\bigg[\displaystyle \int_0^t \|b(X_s)\|^2\,ds \bigg] < \infty,\quad 
        E\bigg[\displaystyle \int_0^t |a(X_s)|\,ds \bigg] < \infty, \\
        E\bigg[\displaystyle\int_0^t \int_{\R^n} \|c(X_{s-})\| \,\nu(ds,dy) \bigg]< \infty,\quad E\bigg[\displaystyle\int_0^t \int_{\R^n} |f(X_{s-},y)| \, \nu(ds,dy) \bigg]< \infty.
\end{align*}
\end{assumption}

Under the above assumptions, we obtain the following Kushner-Stratonovich equation describing the dynamics of $\pi(\phi)$, which is one of the main result in this paper. To this, we denote by $\tilde{\mu}$  the $ \bbF^Y$-compensated random measure $\mu$, given by
\begin{align}
\label{F Y-compensated random measure}
    \tilde \mu (dt,dy) &= \mu(dt,dy)-\nu(dt,dy).
\end{align}

Furthermore, for a bounded, real and measurable function $\phi$ we define the predictable function $S=S(\phi)$ on $\Omega \times \R_{\ge0}\times\R^n$ by
\begin{align}\label{def:S}
    S(\phi)(t,y) := E[\Delta \phi(X_t) |\ccF_{t-}^Y,\Delta Y_{t}=y].
\end{align}
\begin{remark}\label{Remark S}
    To achieve well-posedness of $S$ in the definition \eqref{def:S}, we note that
    the optional projection $\xi:={}^{o,\bbF^Y}(\phi(X))$ is the unique $\bbF^Y$-adapted process such that $\xi_T=E[\phi(X_T)|\ccF^Y_T]$ for all $\bbF^Y$-stopping times $T$. By construction of $\bbF^Y$, $\ccF_t^Y=\sigma(\ccF_{t-}^Y,\Delta Y_t)$, such that 
    $\xi_t=F(t,\Delta Y_t)$, where $F:\Omega \times \R_{\ge 0} \times \R^n$ itself is predictable.
    This is precisely the function which we denote in Equation \eqref{def:S} by $S(\phi)(t,y)$.
\end{remark}

Before we state our first filtering result, we recall that $A_t= \sum_{i=1}^K \ind{T_i \le t}$ denotes the counting process of jump times, with associated measure $dA_t=  \sum_{i=1}^K \delta_{T_i}(dt)$
\begin{theorem}\label{Kushner-Stratonovich}
Assume that Assumption \ref{assumption 1}  holds. Then, for  every  $\phi \in \mathcal{C}_b^{2}(\R^m; \R)$,
    \begin{equation}\label{6}
        \pi_t(\phi)=\pi_0(\phi) + \int_0^t \pi_s(\ccL\phi)\,ds +\int_0^t \pi_{s-}(\ccA\phi)\, dA_s + (S(\phi)*\tilde \mu)_t,\quad   t\ge 0,
    \end{equation}
 \end{theorem}
 
The proof of this result requires some preparation, which we will provide in the following sections. 
The proof itself is relegated to  Section \ref{sec:proof KS}.

Note that the integral with respect to the random measure $\tilde \mu$ in \eqref{6} simplifies to
    \begin{equation*}\label{finite form}
        \pi_t(\phi)=\pi_0(\phi) + \int_0^t \pi_s(\ccL \phi) ds +\sum_{i:T_i \le t} \pi_{{T_i}-}(\ccA \phi) + \sum_{i:T_i \le t} \Delta(S(\phi) * \tilde \mu)_{T_i}
    \end{equation*}
    where the stochastic integral reduces to a finite sum over the predictable jump times $\{T_1,\ldots,T_K\}$

\begin{remark}[Connection to classical Filtering]\label{connection}
For $m=n=1$, when the jump times $T_1,\dots,T_K$ are not predictable but totally inaccessible, the filtering problem reduces to a classical case, see e.g.\ Theorem 3.1 in \cite{ceci2012nonlinear}. First, in this case $A_t=t$. Hence, one can combine $\overline{\ccL}\phi = \ccL \phi + \ccA \phi$ and the Kushner-Stratonovich equation becomes
    \begin{equation*}
        \pi_t(\phi)=\pi_0(\phi) + \int_0^t \pi_s(\overline{\ccL}\phi(X_s))\,ds + (S(\phi)*\tilde \mu)_t,\quad   t\ge 0, 
    \end{equation*}
    Moreover, Equation (3.4) in \cite{ceci2012nonlinear} shows that $S(\phi)$ is replaced by  the well-known form
    \begin{align*}
       \dfrac{d\pi_{t-}(\lambda F\,\phi)}{d\pi_{t-}(\lambda F)}(y)-\pi_{t-}(\phi)+\dfrac{d\pi_{t-}(\bar L\phi)}{d\pi_{t-}(\lambda F)}(y),
    \end{align*}
    where $\lambda$ is the jump intensity, $F$ the jumps size distribution and $\bar L \varphi$ accounts for common jumps between signal and observation; for details we refer to \cite{ceci2012nonlinear}. This term corresponds to 
        \begin{align*}
        S(\phi)(t,y) =E[\phi(X_t) |\ccF_{t-}^Y,\Delta Y_{t}=y]-\pi_{t-}(\phi)
    \end{align*}
    in Equation \eqref{6}. From this it can be seen that the fundamental difference between these two cases is that for predictable jump times, the update is given by a conditional expectation at these times, whereas for totally inaccessible jump times, the update involves intensity-weighted averages due to the uncertainty about jump timing.
\end{remark}

\section{Examples}\label{sec3: motivating examples}

To illustrate the flexibility of our framework, we provide a number of examples. We first cover scheduled (i.e.\ deterministic) jumps, which cover discrete time as a special case and afterwards provide also examples for predictable but not deterministic times. 

\subsection{An Ornstein-Uhlenbeck process with deterministic jumps and the related Kalman filter}\label{Ornstein-Uhlenbeck process}

In many applications, in particular short-horizon, the scheduled times we are looking at are already fixed at initiation. These jumps therefore can be considered as deterministic, which will be covered in our first main example. 

Assume that $X$ is given by the unique strong solution of the SDE
\begin{align}
	\label{eq:OU}
	dX_t = -\lambda X_t \, dt + \sigma_X \, dB_t + \displaystyle \sum_{i=1}^K X_{t-}\xi_i \delta_{T_i}(dt),
    \qquad t \ge 0
\end{align}
where $\lambda$ is the mean-reversion parameter (and the mean-reversion level for simplicity set to zero), $\sigma_X$ is the volatility and $\xi_i$ represent the jump heights at the deterministic times $0<T_1< \dots< T_K$. 

This process is not a jump-diffusion, since the jumps are not driven by a L\'evy process, or more technically the jumps are not totally inaccessible. However, the process has still a high degree of tractability since it is an affine semimartingale in the sense of \cite{keller2019affine}. 

    \begin{figure}[!t]
    \centering    \includegraphics[width=14cm]{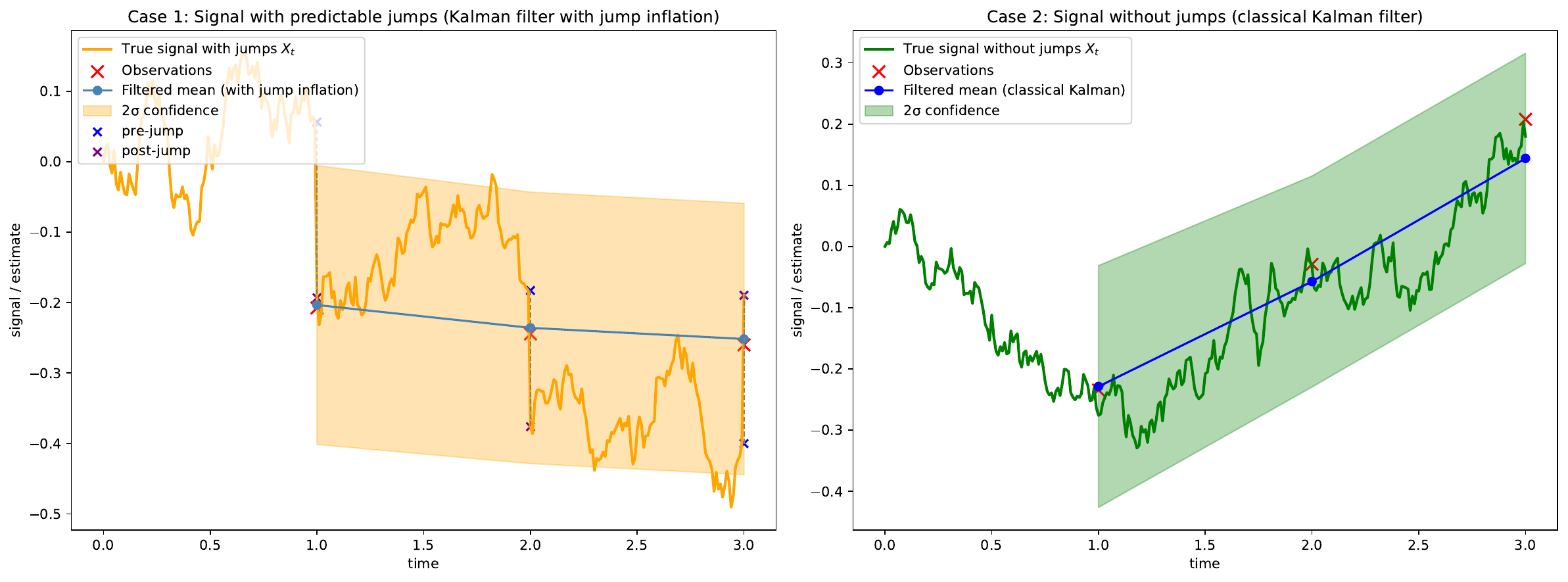}
    \caption{Comparison between Kalman filter case with and no jump }
    \label{fig:example 2}
    \end{figure}

\subsection{A Kalman filter with predictable jumps}\label{sec:Kalman filter}

To reach a Gaussian setting, we additionally assume that  the jump sizes $\xi_1,\dots,\xi_K$ are {i.i.d. $\ccN(0,Q)$} and that, for matrices $A \in \R^{n\times m}$ and $C \in \R^{n\times n},$    
$$Y_t = \sum_{i=1}^K \ind{T_i \le t}\,  (A X_{T_i-} - CY_{T_i-}+ \eta_i) , \quad t \ge 0,$$
    and $\eta_1,\dots,\eta_K$ are {i.i.d. $\ccN(0,R)$}, independent of $B$ and $(\xi_i)$, such that we arrive at an extension of the classical Kalman filter, as we will show in the following. 

Filtering in this setting proceeds in two steps: classical Kalman filtering between the jumping times and updating explicitly at the jumping times. For the first step, observe that in between observations we perform a classical smoothing step, such that the conditional distribution of $X_t$ given $\ccF_t^Y$ is Gaussian with mean $m_t$ and variance $P_t$, where for $t \in (T_{i-1},T_i]$, 

$$m_t = m_{T_{i}}e^{-\lambda(t-T_{i})},\quad P_t= \dfrac{\sigma^2_X}{2\lambda}(1-e^{-\lambda(t-T_{i})})+ P_{T_{i}} e^{-2\lambda(t-T_{i})}.$$

At observation times we simply do a Bayesian updating - much like a Kalman filter in discrete time \citep{jazwinski2007stochastic}. More precisely, we obtain that 
$$m_{T_{i}}= m_{T_{i}-}+K_iv_i, \quad 
P_{T_{i}}=\hat P_{T_{i}}-K_iA \hat P_{T_{i}} $$
where $v_i=\Delta Y_{T_i}-Am_{T_{i}^{-}}$, $\hat P_{T_i}=P_{T_i-}+Q$ and 
 the classical Kalman gain is given by $K_i =\hat P_{T_i}A\, (S_i)^{-1}$, and $S_i = A^2\hat P_{T_i} + R$.

Compared to the classical Kalman filter, we obtain a discrete-time-like updating step at the jump times. In our setting here, we smooth in between the observation times and then update at the observation time, taking the variance of the jump sizes, $Q$, into account.

\subsection{Medical examinations: predictable times in a clinical context}\label{medical examination}

Following \cite{hackenberg2022deep}, medical examinations often lead to a small data context where the precise treatment of the dynamics of the random times become important. In many cases, medical examinations are scheduled sequentially, often depending on the outcome of the preceding examination. In this setting, the examination times are \emph{predictable}.

As a quantitative example we study the area of health monitoring.
In this field, a question of interest is whether similar dynamical patterns are shared across a heterogeneous patient cohort, see for example \cite{li2014physiological}. The following, stylised example illustrates how such a setting could be incorporated in the framework proposed here.

The main quantity of interest is the patients \emph{health state}, which is not directly observable and therefore will be modelled as  signal $X$. Larger values of $X$ indicate better health state, bounded below by zero. Let us assume that 
\begin{align}
	dX_t = \alpha X_t \,dt + \beta X_t \, dB_t + 
	\displaystyle \sum_{i=1}^K X_{t-}\xi_i \delta_{T_i}(dt), \qquad t \ge 0,
\end{align}
with $\alpha, \beta, x_0 >0$. The parameter $\alpha$ can be used to model the deterioration ($\alpha <0$) or improvement ($\alpha >0$) of the health state.
In contrast to the Ornstein-Uhlenbeck example, $X$ behaves like a geometric Brownian motion in between the jump times and therefore remains non-negative. If  the $(\xi_i)$ were i.i.d.,\ it would be necessary to require $\xi_i \ge 0$ to keep non-negativity. However, if $\xi_i$ is allowed to depend on $X_{T_i-}$, we have much more flexibility. A particularly tractable example is when $\xi_i=X_{T_i-} \cdot \tilde \xi_i$ with $\tilde \xi_i \ge 0$ which directly would imply non-negativity of $X$, since then $X$ turns out to be a stochastic exponential, see for example Section I.4f in \cite{jacod2013limit}. This also includes the exponential of the Ornstein-Uhlenbeck process from Equation~ \eqref{eq:OU}.

The observation is given by the process $Y$  which is the (possibly at high frequency\footnote{A continuous observation could be included here without much additional difficulty(see \cite{bain2009crisan} for details on continuous-time filtering) -- which is however not our focus here.}) monitored \emph{health score}. 
The \emph{observation times} are placed on a discrete grid $0=s_0<s_1< \cdots < s_\ell$ for some fixed $\ell \in \N$ and we assume that $Y_{s_j}$ is as in Equation \eqref{eq:Y}.

In addition, the predictable times $(T_i)$ depend on the health score and model thresholds where interventions are taken (and therefore also influence $X$). More precisely, we assume that there is a scale of decreasing thresholds $\theta_K>\cdots > \theta_1$ and     $$T_i = \inf\{s_j \ge 0: Y_{s_j} \le \theta_i\}.$$
If the threshold $\theta_i$ is never reached, we set $T_i=+\infty$ by convention, meaning no intervention occurs at level $i$.
Of course, if the health score improves, there could be a reset in between, which we do not include for simplicity. 
If $\xi_i$ are log-normal and $Y$ depends linearly on $\log X$, then a Kalman-filter like setting is recovered here as well. Mostly, however, the situation will be more involved.

\subsection{Neural jump ODEs}\label{NODEs}
Consider a filtering problem where the signal process $X$ does not jump, i.e.\
\begin{equation*}
dX_t = a(X_t) \, dt + b(X_t) \, dB_t, \quad X_0 =x_0 \in \R^m.
\end{equation*}
 Observations occur at predictable times $T_1, \ldots, T_K$ with additive Gaussian noise, where we additionally assume $f(x,y)=h(x)$, i.e.\ the observation function does not depend on the pre-observation value of $Y$. More precisely,
\begin{equation*}
Y_t = \sum_{i=1}^K \ind{T_i \le t}(h(X_{T_i-}) + \eta_i),\quad \eta_i \sim \ccN(0,R)
\end{equation*}
where $\eta_i$ are i.i.d observation errors and  $h: \R^m \to \R^n$ is the observation function. 

The goal is to compute $\pi_t(\phi) = E[\phi(X_t)|\ccF^Y_t]$ for bounded test functions $\phi \in C_b^2(\R^m;\R)$. Since $c \equiv 0$, the term $\ccA \phi$ vanishes and our Theorem~\ref{Kushner-Stratonovich} applies to this setting and yields the explicit recursive equation
\begin{equation*}
        \pi_t(\phi)=\pi_0(\phi) + \int_0^t \pi_s(\ccL\phi)\,ds  + (S(\phi)*\tilde \mu)_t,\quad   t\ge 0.
    \end{equation*}
where $\ccL \phi(x)= a(x)\cdot\nabla \phi(x) + \frac{1}{2}\text{tr}(\sigma_X(x)\sigma_X(x)^\top \nabla^2 \phi(x)),$     
and the innovation term $(S(\phi)*\tilde \mu)_t$ captures the Bayesian update at each predictable observation time $T_i$. 

When the signal dynamics are unknown or high-dimensional, computing $\pi_t(\phi)$ analytically iss generally intractable. In such cases, following the Input-Output Neural Jump ODE (IO-NJODE) framework of \citet{heiss2025nonparametric}, with $U=Y$ as input and $V=X$ as output process, one can approximate the conditional expectation $E[X_t|\ccF^Y_t]=\pi_t(\text{id})$ by a neural architecture where a hidden state $H_t \in \R^{m}$ evolves according to:
\begin{align*}
H_0 &= \rho_{\theta_2}(x_0), \nonumber \\
dH_t &= f_{\theta_1}(H_{t-}, \Delta Y_{\tau(t)}, \tau(t), t-\tau(t)) \, dt + \left(\rho_{\theta_2}(\Delta Y_{t}) - H_{t-}\right) dJ_t'\\
G_t &= g_{\theta_3}(H_t), \label{eq:njode}
\end{align*}
where $f_{\theta_1}, \rho_{\theta_2}$ and $g_{\theta_3}$ are feedforward neural networks with trainable parameters $\theta = (\theta_1, \theta_2, \theta_3)$, $\tau(t) = \max\{T_i : T_i \leq t\}$ is the last observation time, and $J_t' = \sum_{i=1}^K \ind{T_i \leq t}$ counts the observations. The filtered estimate $G$ converges to $E[X_t|\ccF^Y_t]$ in $L^2$ as the network capacity increases, as established in \citet[Section 4]{heiss2025nonparametric}

This makes IO-NJODEs particularly suitable for applications where the signal dynamics are unknown or too complex for analytical treatment, yet training data (either historical or simulated under simplified assumptions) is available. 

\subsection{Credit risk with incomplete information}\label{credit risk}
A natural and economically important application of this framework arises in structural credit-risk modelling. In structural credit risk models such as  \cite{merton1974pricing}, the firm's asset value $V$ is assumed to be observable and default occurs when $V$ falls below the liability threshold $(K_t)_{t \ge 0}$. In practice, however, investors observe the firm value only indirectly through accounting releases, credit-rating updates, stock prices, or other  disclosures. \cite{duffie2001term}  formalized this idea of credit risk under incomplete information. Building on this insight and \cite{frey2009pricing}, one of our forthcoming work develops a structural model in which investors filter the firm's  asset process when both signal and the observation may jump at predictable announcement times. In the following, we illustrate how this model can be embedded into our framework. 

Let $X_t = \log V_t,$ $t \ge 0$, denote the unobserved log-asset value. Between scheduled announcement dates the assets evolve as geometric Brownian motion i.e. $dV_t = V_t(\mu_V dt + \sigma_V dB_t)$, so
\begin{align*}
    dX_t = \left(\mu_V - \dfrac{\sigma_V^2}{2}\right) dt + \sigma_V dB_t \quad t \in (T_{i-1},T_i).
\end{align*}
The predictable times $T_1,\ldots,T_K$ represent scheduled corporate events such as quarterly earnings announcements, coupon payments, or regulatory fillings. At these dates both the firm's asset value and the market's information may jump. At each announcement time $T_i$, 
$X_{T_i} = X_{T_i^-} + \tilde c(X_{T_i^-}, Y_{T_i^-}, \xi_i),$ where the jump function $\tilde c$ captures the magnitude of write-downs or revaluations. A convenient specification is 
\begin{align*}
    \tilde c(x,y,\xi) = \log (1-\xi(1+\beta \ind{y<\bar y})),
\end{align*}
with $\xi \in [0,\frac{1}{1+\beta})$ begin a random loss proportion and $\beta \ge 0$ begin an amplification factor reflecting that negative news has greater impact when the sentiment $y$ is already weak.

Investors observe noisy disclosures in the following form:
$$\Delta Y_{T_i}=f(X_{T_i^-}, Y_{T_i^-})+ \eta_i, \qquad \eta_i \sim \ccN(0,\sigma^2_\eta),$$
where $f(x,y)=x+\alpha_y y$ links current fundamentals with past public information. Additional indicators such as dividend payment $d_i$ may also be observed at a subset of the $T_i$, with conditional density $\nu_d(d_i|X_{T_i^-}, Y_{T_i^-})$. In addition, market participants observe default events, discrete disclosures, and cumulative dividend process $D$:
\begin{align}
    \ccF_t^Y := \sigma \big(\ind{\tau \le s}, D_s, Y_s: 0\le s \le t\big) \vee \ccN
\end{align}
where $K$ is the liability barrier. Note that this model  fits  into our  framework (see Theorem \ref{Kushner-Stratonovich}). In particular
\begin{align*}
    a(x) = \mu_V - \dfrac{\sigma_V^2}{2}, \quad b(x) = \sigma_V, \quad c(x)\xi=\tilde c(x,Y_{T_i^-},\xi).
\end{align*}
and the marks $Z_i=(\xi_i,\eta_i)$ are i.i.d and independent of $B$. The corresponding Kushner-Stratonovich equation with predictable jumps yields explicit Bayesian update rules for investor beliefs at each scheduled event.

\section{Proof of the main result}\label{sec4: proof of KS}
The first step in proving  Theorem \ref{Kushner-Stratonovich}, will be to derive suitable martingale representations in our setting. The key is the martingale representation in the observation filtration $\bbF^Y$. Since in classical filtering results, jumps arise at totally inaccessible stopping times, we will need to develop some suitable representation results for a filtration with jumps at predictable times.

\subsection{Martingale representations}
We start with a generalisation of Lemma 3.2 in \cite{frey2012pricing} to our  setting. Recall  the measure $\mu$, decoding the jumps in the observation filtration and introduced in Equation \eqref{eq:mu}, and its  $\bbF^Y$-compensated version $\tilde \mu$, introduced in  Equation \eqref{F Y-compensated random measure}.

\begin{proposition}\label{Martingale w.r.t F Y}
    For each real-valued $\bbF^Y$-martingale $M$, there exists an $\bbF^Y$-predictable function $S: \Omega \times \R_{\ge 0} \times \R^n  \to \R $ satisfying $P\big((S*\nu)_\infty <\infty\big)=1$, such that, $P$-a.s., 
\begin{align}
\label{eq:679} 
	M_t &= M_0 +(S * \tilde \mu)_t, \quad t \ge 0.
\end{align}
\end{proposition}

\begin{proof}
We observe that the random measure $\mu$ defined in Equation \eqref{eq:mu} is a multivariate point process in the sense of Definition III.1.23 in \cite{jacod1975multivariate}. Moreover, filtration $\bbF^Y$ satisfies assumption III.1.25 therein. Hence, we may apply Theorem III.4.37 and obtain representation \eqref{eq:679} where $|W|*\mu$ is locally integrable. 
\end{proof}

We now present a specific result concerning martingales with respect to the filtration $\bbF$. For every $\phi \in \mathcal{C}_b^{2}(\R^m; \R)$, define the process $M^\phi$ by:
\begin{equation} \label{Mat1}
    M^\phi_t = \phi (X_{t}) - \phi (X_0) - \int_0^t \ccL\phi(X_{s})ds - 
    \int_0^t \ccA \phi(X_{s-}) dA_s.
\end{equation} 

\begin{proposition}\label{Martingale w.r.t F}
    Suppose Assumption \ref{assumption 1} holds and let $\phi \in \mathcal{C}_b^{2} ( \R^m;\R)$. Then the process $M^\phi$, defined in \eqref{Mat1},
is an $\bbF$-martingale.
\end{proposition}

\begin{proof}
Starting from Equation \eqref{sys: 2}, we apply  It\^o's formula to $\phi(X)$, which gives
   \begin{align*}
    \phi(X_t)&=\phi(X_0)+ \int_0^t \ccL \phi(X_s) \, ds + \int_0^t \nabla_x \phi(X_s)^\top b(X_s)\, dB_s \\
    & \qquad \qquad+\int_0^t\int_{\R^{m+n}}  \big(\phi(X_{s-}+c(X_{s-})\xi) - \phi(X_{s-})\big)\, \psi(ds,dz).
\end{align*}
Hence, using the representation of $\ccA$ in Equation \eqref{eq: ccA}, 
we obtain
\begin{equation}\label{Martingale expression}
    M^\phi_t =  \int_0^t \nabla_x \phi(X_s)^\top b(X_s)\, dB_s +\int_0^t\int_{\R^{m+n}}  \big(\phi(X_{s-}+c(X_{s-})\xi) - \phi(X_{s-})\big)\, q(ds,dz),
\end{equation}
where $q(ds,dz)=\psi(ds,dz)-F_Z(dz)dA_s$ is the compensated random measure. 
Note that the first summand is a square-integrable martingale since $\phi \in C_b^{2}$ implies $\|\nabla_x\phi\|_\infty <\infty$ and\\ $E[\int_0^t \|b(X_s)\|^2ds]<\infty$ by Assumption \ref{assumption 1}.

Since by assumption $|\phi(\cdot)| \le C < \infty$, we obtain that $\tilde W:=(\phi(X_{s-}+c(X_{s-})\xi) - \phi(X_{s-}))$ satisfies  $|W| \le 2C$, and hence
$$
(\tilde W*\psi)_\infty \le 2C\cdot K
$$
which implies also that $(|\tilde W| *q)_\infty < \infty$ and hence, the second summand is bounded and therefore also a martingale.
\end{proof}

\subsection{Proof of Theorem \ref{Kushner-Stratonovich}}\label{sec:proof KS}

We start with a preliminary lemma, which states that all terms on the right hand side in Equation  \eqref{6} are finite.

\begin{lemma}\label{lem:conditions}
	Under Assumption \ref{assumption 1}, for every $\phi \in \mathcal{C}_b^{2}(\R^m; \R)$, the following holds:
	\begin{align*}
		E\bigg[\int_0^T \lvert \ccL \phi(X_t)\rvert dt \bigg] +
		E\bigg[\int_0^T \lvert \pi_t(\ccL \phi)\rvert dt \bigg] & < \infty, \qquad T \ge 0,\\
		E\bigg[\int_0^T \lvert \ccA \phi(X_{t-})\rvert dA_t \bigg] + 
		E\bigg[\int_0^T \lvert \pi_{t-}(\ccA \phi) \rvert dA_t \bigg] & < \infty, \qquad T \ge 0,\\
		E\big[\displaystyle (|S(\phi)|*\nu)_\infty\big] & < \infty.
	\end{align*}
\end{lemma} 
\begin{proof}
	Since $\phi \in \mathcal{C}_b^{2}$, $\nicefrac{\partial \phi}{\partial x}$ is bounded and therefore $|\ccL \phi(X_t)| \le C_1 (\sum_{i=1}^m|a_i(X_t)| + \sum_{i,j=1}^m \sigma_{ij} (X_t))$ for some $C_1 >0$ and the first claim readily follows from Assumption \ref{assumption 1}.

    For the second claim, observe that  $\phi(x_1)-\phi(x_2) \le 2 \lVert \phi \rVert_\infty$ for any $x_1,x_2 \in \R^m$, such that
 $$|\ccA \phi(X_{t-})|=\bigg|\int_{\R^{m+n}} \big(\phi(X_{t-}+c(X_{t-})\xi) - \phi(X_{t-})\big)F_Z(d\xi \otimes d\eta) \bigg|\le 2 \lVert \phi \rVert_\infty$$
and hence
$$\int_0^T \lvert \ccA \phi(X_{t-})\rvert dA_t = \sum_{i=1}^K |\ccA \phi(X_{T_{i}^{-}})| \le 2 K\lVert \phi \rVert_\infty.$$
Using Jensen's inequality, we obtain $|\pi_{t-}(\ccA \phi)|\le \pi_{t-}(|\ccA \phi|) \le 2 \lVert \phi \rVert_\infty,$ such that the second claim follows.

	For the third claim, using  \eqref{def:S}, and again $|\Delta \phi| \le 2 \lVert \phi \rVert_\infty $, we obtain that
$$|S(\phi)(t,y)| \le  E[|\Delta \phi(X_t)|\, |\ccF_{t-}^Y,\Delta Y_{t}=y] \le 2 \lVert \phi \rVert_\infty $$
and hence $E\big[\displaystyle (|S(\phi)|*\nu)_\infty\big]\le 2 K \lVert \phi \rVert_\infty$.
\end{proof}

The following lemmas are well-known results which we will use frequently. For the convenience of the reader we provide a short proof for each. To this end, we will write ${}^o\!H$ for the $\bbF^Y$-optional projection of a progressively measurable process $H$, satisfying $E[|H_t|]<\infty$ for all $t$, defined as the unique optional process such that, for any $\bbF^Y$-stopping time $T$, $({}^o\!H)_T=E[H_T|\ccF_T^Y]$ a.s. on $\{T <\infty\}$.
\begin{lemma}\label{dif:clarrification of difference martingale}
    Consider an $\bbF^Y$-progressive process $H$  satisfying
    $
       E[\int_0^t \lvert H_s \rvert \, ds ] < \infty $ for all $t \ge 0$. 
    Then there exists a c\`adl\`ag and adapted $\bbF^Y$-martingale $M$ which satisfies 
    $$
        M_t =E\Big[ \int_0^t H_s \, ds|\ccF_t^Y\Big]-\int_0^t ({}^o\!H)_s \, ds, \qquad t \ge 0.
    $$ 
\end{lemma}
\begin{proof}
By Theorem V.5.25 in \cite{he1992semimartingale}, the optional projection of $\int_0^\cdot H_s \,ds$ is given by $\int_0^\cdot ({}^o\!H)_s \,ds$. The claim now follows from Corollary V.5.31.2) by considering $\Ind_{[0,t]}H$ with  $t<\infty$. %
\end{proof}

\begin{lemma}\label{dif: clarrification of finite variation term}
    Let $N$ be $\bbF^Y$-adapted, bounded and $\phi \in \mathcal{C}_b^{2}(\R^m;\R)$. Then the process
    $$M_t := E\bigg[ \int_0^t N_{s-} \ccA \phi(X_{s-}) dA_s\,\Big|\,\ccF_t^Y\bigg]-\int_0^t N_{s-} E[ \ccA \phi(X_{s-})\,|\,\ccF_{s-}^Y] \,dA_s, \quad \quad t \ge 0$$
    is an $\bbF^Y$-martingale.
\end{lemma}
\begin{proof}
    Since $dA_t = \sum_{i=1}^K \delta_{T_i}(dt)$,
    $$M_t = \sum_{i=1}^K \ind{T_i \le t} \bigg( E[N_{T_{i}^{-}}\ccA \phi(X_{T_{i}^{-}}) |\ccF_t^Y]- N_{T_{i}^{-}}E[\ccA \phi(X_{T_{i}^{-}}) |\ccF_{T_{i}^{-}}^Y] \bigg).$$
    
    Set $X_{T_{i}^{-}}=N_{T_{i}^{-}}\ccA \phi(X_{T_{i}^{-}})$. Using the tower rule, we obtain that for $0 \le s \le t$,
    \begin{align*}
        E[M_t -M_s|\ccF_s^Y] &= \sum_{i:T_i \in (s,t]} E \bigg[E[X_{T_{i}^{-}} |\ccF_t^Y] - N_{T_{i}^{-}}E[\ccA \phi(X_{T_{i}^{-}}) |\ccF_{T_{i}^{-}}^Y] \,\Big| \, \ccF_s^Y \bigg] \\
        &= \sum_{i:T_i \in (s,t]} \bigg( E[X_{T_{i}^{-}} |\ccF_s^Y]-E \big[ N_{T_{i}^{-}}E[\ccA \phi(X_{T_{i}^{-}}) |\ccF_{T_{i}^{-}}^Y] \big|\ccF_s^Y \big]\bigg).
    \end{align*}
    In addition we have that
    \begin{align*}
        E[X_{T_{i}^{-}} |\ccF_s^Y]= E \big[ E[X_{T_{i}^{-}} |\ccF_{T_{i}^{-}}^Y]\big|\ccF_s^Y \big]= E \big[N_{T_{i}^{-}}E[\ccA \phi_{T_{i}^{-}} |\ccF_{T_{i}^{-}}^Y]\big|\ccF_s^Y \big],
    \end{align*}
    such that $E[M_t -M_s|\ccF_s^Y]=0.$
\end{proof}

\bigskip

We are now ready for the proof of the Kushner-Stratonovich equation.

\begin{proof}[Proof of Theorem \ref{Kushner-Stratonovich}] 
Consider $\phi \in C_b^{2}$, and recall from Proposition \ref{Martingale w.r.t F} that $M^\phi$, given in \eqref{Mat1} is a $\bbF$-martingale. For simplicity, we  write short $\phi_t=\phi(X_t)$.
Taking conditional expectations on \eqref{Mat1} we therefore obtain 
\begin{align}
   \label{p2} E[\phi_t|\ccF_t^Y]
 &=E[\phi_0|\ccF_t^Y] + E\bigg[\int_0^t\ccL \phi_s ds\Big|\ccF_t^Y\bigg]
     + E\bigg[\int_0^t  \ccA \phi_{s-}dA_s \Big|\ccF_t^Y\bigg] \nonumber \\
     & + E[M^\phi_t|\ccF_t^Y].
\end{align}
Under Assumption \ref{assumption 1} we may apply 
Lemma \ref{dif:clarrification of difference martingale} due to Lemma \ref{lem:conditions} and obtain
\begin{align}
	E\bigg[\int_0^t\ccL \phi_s \, ds\Big|\ccF_t^Y\bigg] &=
	\int_0^tE\big[\ccL \phi_s ds|\ccF_s^Y\big] \, ds + m_t^1 
	=\int_0^t \pi_s\big(\ccL \phi \big) \, ds + m_t^1  , \qquad t \ge 0
\end{align}
with an $\bbF^Y$-martingale $m^1$.
Similarly, we apply 
Lemma \ref{dif: clarrification of finite variation term} with the aid of Lemma \ref{lem:conditions} and obtain
\begin{align}
	E\bigg[\int_0^t  \ccA \phi_{s-}\, dA_s \Big|\ccF_t^Y\bigg] &=
	\int_0^t \pi_{s-}\big(  \ccA \phi \big) \, dA_s + m_t^2, \qquad t  \ge 0,
\end{align}
with an $\bbF^Y$-martingale $m^2$. Moreover, since $X_0=x_0$ is deterministic, we have $E[\phi_0|\ccF_t^Y]=\phi_0=\pi_0(\phi)$. 
We add the $\bbF^Y$-martingales to $M_t:= E[M^\phi_t|\ccF_t^Y] + m_t^1+m_t^2$, $t \ge 0$ and, by Proposition \ref{Martingale w.r.t F Y},
there exists a function $S$, such that $ M= (S*\tilde \mu)$. 
Altogether, we obtain that
\begin{align}
   \label{eq:temp pi} 
   \pi_t(\phi) &=\pi_0(\phi) + \int_0^t \pi_s\big(\ccL \phi \big) \, ds 
     + \int_0^t \pi_{s-}\big(  \ccA \phi \big) \, dA_s + (S*\tilde \mu)_t, \qquad t \ge 0.
\end{align}

The next step is to obtain the desired representation of $S$. 
To achieve this, we consider an $\bbF^Y$-martingale $N=(T*\tilde \mu)$ with a bounded and predictable function $T:\Omega \times \R_{\ge 0} \times \R^n\to \R$.
 By the product rule, we have
\begin{align}
   \nonumber \phi_t N_t &= (\phi_- \cdot N)_t + (N_- \cdot \phi)_t +  [N,\phi]_t\\
  \label{rule 1} &= (\phi_- \cdot T * \tilde \mu)_t + \int_0^t N_{s-} \ccL \phi_s ds + \int_0^t N_{s-} \ccA \phi_s dA_s\\
   \nonumber &  + \int_0^t N_{s-} dM_s^\phi + \sum_{s\le t}\Delta \phi_s \Delta N_s.
\end{align}

We start by considering the last term. Note that $\Delta \phi_s \Delta N_s=0$ for $s \not \in\{T_1,\dots,T_K\}$, such that we may directly concentrate on $ \Delta \phi_{T_i} \Delta N_{T_i}$. 
We need to compute its projection to $\bbF^Y$. Denote 
\begin{align}\label{def:F1}
	F^1(T_i,\Delta Y_{T_i}) &:= 
	E[\Delta \phi_{T_i} \Delta N_{T_i}|\ccF_{T_i^-}^Y,\Delta Y_{T_i}], 
\end{align}
and  let $F^1(t,y) := \sum_{i=1}^K \ind{t=T_i} F^1(t,y)$,
such that the function $F^1(s,y)$ is $\bbF^Y$-predictable. 

Hence,
\begin{align*}
	E[\Delta \phi_s \Delta N_s|\ccF_s^Y] &= \int_{\R^n} F^1(s,y) \, \mu(\{s\},dy)
\end{align*}
and we obtain that for all $t \ge 0$
\begin{align}\label{p4}
    E[\phi_tN_t|\ccF_t^Y] &= \int_0^t N_{s-} \pi_{s-}(\ccL \phi)\,ds + \int_0^t N_{s-} \pi_{s-}(\ccA \phi)\,dA_s + (F^1 * \mu)_t  + m^3_t,
\end{align}
with an $\bbF^Y$-martingale $m^3$.

On the other hand, 
\begin{align}\label{p5}
    \pi_t(\phi)N_t &= (N_- \cdot \pi(\phi))_t + (\pi_-(\phi) \cdot N)_t + [N,\pi(\phi)]_t\\
 \nonumber   &= \int_0^t N_{s-} \pi_{s-}(\ccL \phi)ds + \int_0^t N_{s-} \pi_{s-}(\ccA \phi)dA_s + \int_0^t \int_{\R^n} N_{s-} S(s,y) \, \tilde\mu(ds,dy) \\
\nonumber &+ \sum_{s\le t}\Delta \pi_s(\phi) \Delta N_s 
           + \int_0^t \int_{\R^n} \pi_{s-}(\phi) T(s,y)\, \tilde\mu(ds,dy).
\end{align}
This time, we compute from Equation \eqref{eq:temp pi} and the definition of $N$ that for $t \in \{T_1,\dots,T_k\}$,
\begin{align*}
	\Delta \pi_t(\phi) \Delta N_t &= 
	\pi_{t-}\big(  \ccA \phi) \big) \cdot \Delta N_{t}  
    + \Delta (S*\tilde \mu)_t \cdot \Delta (T*\tilde \mu)_t. 
\end{align*}
Note that, again for $t \in \{T_1,\dots,T_k\}$,
\begin{align*}
    \lefteqn{\Delta (S*\tilde \mu)_t \cdot \Delta (T*\tilde \mu)_t = \int_{\R^n} S(t,y) \tilde \mu(\{t\},dy) \cdot  \int_{\R^n} T(t,y) \tilde \mu(\{t\},dy) } \quad \\
    &= \int_{\R^n} S(t,y)  \mu(\{t\},dy) \cdot  \int_{\R^n} T(t,y)  \mu(\{t\},dy) 
    -\int_{\R^n} S(t,y)  \mu(\{t\},dy) \cdot  \int_{\R^n} T(t,y)  \nu(\{t\},dy) \\
    &- \int_{\R^n} S(t,y)  \nu(\{t\},dy) \cdot  \int_{\R^n} T(t,y)  \mu(\{t\},dy) 
    + \int_{\R^n} S(t,y)  \nu(\{t\},dy) \cdot  \int_{\R^n} T(t,y)  \nu(\{t\},dy)\\
    &= \Delta (S \cdot T * \tilde \mu)_t + \Delta (S \cdot T *  \nu)_t \\
    &-\int_{\R^n} S(t,y)  \tilde \mu(\{t\},dy) \cdot  \int_{\R^n} T(t,y)  \nu(\{t\},dy)  -\int_{\R^n} S(t,y)  \nu(\{t\},dy) \cdot  \int_{\R^n} T(t,y)  \nu(\{t\},dy)\\
    &- \int_{\R^n} S(t,y)  \nu(\{t\},dy) \cdot  \int_{\R^n} T(t,y)  \tilde \mu(\{t\},dy) - \int_{\R^n} S(t,y)  \nu(\{t\},dy) \cdot  \int_{\R^n} T(t,y)  \nu(\{t\},dy) \\
    &+ \int_{\R^n} S(t,y)  \nu(\{t\},dy) \cdot  \int_{\R^n} T(t,y)  \nu(\{t\},dy)\\
    &= \Delta m^4_t + \Delta (S \cdot T *  \nu)_t- \big( \Delta (S*\nu)_t \cdot \Delta(T*\nu)_t \big)
\end{align*}
with martingale part $m^4$, such that
\begin{align}
	\sum_{s \le t}\Delta \pi_s(\phi) \Delta N_s &= 
	(\pi_{-}(  \ccA \phi)\cdot N \big)_t   + m^4_t 
    +  (S \cdot T *  \nu)_t- \sum_{s \le t} \Delta (S*\nu)_s \cdot \Delta(T*\nu)_s. \label{temp1092}
\end{align}
The key observation is that, since $N$ is $\bbF^Y$-measurable, $E[\phi_tN_t|\ccF_t^Y] = E[\phi_t|\ccF_t^Y]N_t=\pi_t(\phi)N_t$, such that representations \eqref{p4} and \eqref{p5} coincide. Since both processes are special semimartingales, also their semimartingale characteristics coincide and in particular their finite variation parts. From Equations \eqref{p4} and \eqref{temp1092} we therefore obtain the equality
\begin{align}
   (F^1 * \mu)_t  &= (S \cdot T *  \nu)_t- \sum_{s \le t} \Delta (S*\nu)_s \cdot \Delta(T*\nu)_s, \quad t \ge 0.
   \label{temp1094}
\end{align}
In the following, we derive a precise representation of the left hand side, following the definition according to \eqref{p4}. To this end, we note that for any of the predictable times $t=T_1,\dots, t=T_n,$
\begin{align*}
    \Delta (F^1 * \mu)_t &=\Delta N_{t} \cdot  E[\Delta \phi_{t} |\ccF_{t-}^Y,\Delta Y_{t}] \\
    &= \Delta (T*\tilde \mu)_t \cdot \Delta (F^2*\mu)_t,
\end{align*}
with $F^2(t,y) :=  E[\Delta \phi_{t} |\ccF_{t-}^Y,\Delta Y_{t}=y].$ Now we can proceed as above, and obtain that
\begin{align*}
   \Delta (T*\tilde \mu)_t \cdot \Delta (F^2*\mu)_t &=   
   \Delta (T* \mu)_t \cdot \Delta (F^2*\mu)_t - \Delta (T* \nu)_t \cdot \Delta (F^2*\mu)_t \\
   &=\Delta (T \cdot F^2*\mu)_t - \Delta (T* \nu)_t \cdot \Delta (F^2*\tilde \mu)_t
   - \Delta (T* \nu)_t \cdot \Delta (F^2*\nu)_t \\
   &= \Delta (T \cdot F^2*\nu)_t- \Delta (T* \nu)_t \cdot \Delta (F^2*\nu)_t +\Delta m^5_t,
   \end{align*}
with an $\bbF^Y$-martingale $m^5$. Hence, from Equation \eqref{temp1094},
\begin{align}
      (T \cdot F^2*\nu)_t- \sum_{s \le t}\Delta (T* \nu)_s \cdot \Delta (F^2*\nu)_s
     &=(S \cdot T *  \nu)_t- \sum_{s \le t} \Delta (S*\nu)_s \cdot \Delta(T*\nu)_s.
\end{align}
Since $T$ was arbitrary, we recover $S=F^2$ and the proof is finished.
\end{proof}

\section{The Zakai Equation}\label{sec5: zakai formulation}
In nonlinear filtering, the Zakai equation provides a \emph{linear} stochastic partial differential equation which is often to treat. The key step for this is to study  the unnormalised conditional distribution instead of the conditional distribution (as was done in the Kushner-Stratonovich equation in Theorem \ref{Kushner-Stratonovich}).

More precisely, the unnormalised measure-valued process $\rho$ allows to recover the conditional probability $\pi$ via the so-called Kallianpur-Striebel equation: for $t \ge 0$,
\begin{align}\label{Striebel}
    \pi_t(\phi)=\dfrac{\rho_t(\phi)}{\rho_t(1)}, \qquad \phi \in \mathcal{C}_b^{2}( \R^m;\R)
\end{align}
and process $\rho$ satisfies a linear equation, which we will show below. This linearization plays a fundamental role in the analysis of filtering problems, and its derivation relies on a change of measure, often referred to as the \emph{change of probability measure method} \citep{bain2009crisan} or as the \emph{reference probability method} \citep{zeitouni1986reference}.

\subsection{A change of measure}
To obtain the unnormalised distribution we perform a suitable change of measure. To this end,  
consider a density process $Z$, i.e.\ a strictly positive $(P, \bbF^Y)$-martingale with $E_P[Z_0]=1$, and 
introduce the probability measure $P'$ on any $(\Omega,\ccF_t)$, $t \ge 0$  via 
$$dP'= Z_t\,dP.$$
This allows us to define the unnormalised conditional distribution $\rho$   by 
\begin{align}
    \rho_t(\phi):= E_{P'}[\phi(X_t)Z^{-1}_t \mid \ccF_t^Y],\qquad \phi \in \mathcal{C}_b^{2}(\R^m;\R), \qquad t \ge 0.
\end{align} 
\begin{proposition}\label{thm: girsanov}
    Let $\Gamma=\Gamma(\omega;t,y)$ be an $\bbF^Y$-predictable function such that $E[(e^\Gamma*\nu)_t]< \infty$
    and 
\begin{align}\label{positivity condition}
    e^{\Gamma(t, y)} - \Delta \tilde B_t > 0, \qquad t \ge 0, 
\end{align}
where
    $ \tilde B := (e^\Gamma -1)*\nu. $
Furthermore, define 
    $
        M:= (e^\Gamma-1)*\tilde \mu.
    $
    Then $Z:=\ccE(M)$ is a strictly positive $(\bbF^Y,P)$-martingale 
     whenever $E[Z_t] = 1$ for all $t \ge 0$.
    Moreover, there exists a locally equivalent measure $P'$, such that
    $$
    dP'_t  = Z_t \, dP_t, \quad t \geq 0,
    $$
    and under $P'$, the compensator of  $\mu$ is given by
    \begin{align}\label{compensator under P'}
    \nu'(dt, dy) = (e^{\Gamma(t, y)}-\Delta \tilde B_t) \nu(dt, dy).
    \end{align}
\end{proposition}

A version of this result is given in \citet[Theorem 3.12]{bjork1997bond}, however in the case where intensities exist. 
Note that in the classical case with totally inaccessible jumps the process $\tilde B$ is absolutely continuous and hence $\Delta \tilde B=0$. Then Equation \eqref{compensator under P'} simplifies to the familiar expression
\begin{align*}
    \nu'(dt, dy) = e^{\Gamma(t, y)} \nu(dt, dy).
    \end{align*}

\begin{proof}[Proof of Proposition \ref{thm: girsanov}]
    First, note that by definition
    \begin{align*}
        \Delta Z_t = Z_{t-}\Delta M_t &=Z_{t-}\bigg(\int_{\R^n} (e^{\Gamma(t, y)} - 1) (\mu(\{t\}, dy)-\nu(\{t\}, dy))\bigg)
        =Z_{t-}\big(e^{\Gamma(t, \Delta Y_t)}-1- \Delta \tilde B_t\big) 
    \end{align*}
    and hence 
    \begin{align}\label{eq:dynZ}
        Z_t=Z_{t-}\big(e^{\Gamma(t, \Delta Y_t)}-\Delta \tilde B_t \big).
    \end{align} Therefore, by condition \eqref{positivity condition}, $Z_t >0$ for all $t$. By construction, $Z$ is the Dol\'ean-Dade exponential of the purely discontinuous local martingale $M$ and hence,  $Z$ is a strictly positive $(\bbF^Y,P)$-local martingale (see \cite[Theorem I.4.61]{jacod2013limit}). The additional condition $E[Z_t] = 1$ ensures it is a true martingale, allowing us to define $P'_t$ by $dP'_t=Z_t\,dP$ on $(\Omega,\ccF_t)$ for $t\ge 0$. This allows us to define a unique $P'$ on $\cup_{t \ge 0}\ccF_t$. This may be extended naturally to $(\Omega, \ccF)$, where however equivalence may not hold.
    Since $Z>0$, $P'$ is always \emph{locally} equivalent to $P$.

    Recall $\tilde \ccP=\ccP \otimes \ccB(\R^n)$ with the already introduced predictable $\sigma$-field $\ccP$. For any nonnegative $\tilde \ccP$-measurable function H and bounded stopping time $T$, by the martingale property of $Z$, 
    \begin{align*}
        E_P\bigg[Z_T\int_0^T \int_{\R^n} H(s,y)\mu(ds,dy)\bigg]&= E_P\bigg[Z_T \sum_{i=1}^K  \ind{T_i \le T} H(T_i,\Delta Y_{T_i}) \bigg] \\
        &=  \sum_{i=1}^K E_P\bigg[E_P[Z_T| \ccF_{T_i}^Y] \,\ind{T_i \le T} H(T_i,\Delta Y_{T_i}) \bigg]\\
        &= \sum_{i=1}^K E_P\bigg[\ind{T_i \le T}Z_{T_i}  H(T_i,\Delta Y_{T_i}) \bigg] \\
        &=\sum_{i=1}^K E_P\bigg[\ind{T_i \le T} Z_{T_i-}(e^{\Gamma(T_i, \Delta Y_{T_i})}-\Delta \tilde B_{T_i}) H(T_i,\Delta Y_{T_i}) \bigg]\\
        &=E_P\bigg[\int_0^T \int_{\R^n} Z_{s-}(e^{\Gamma(s, y)}-\Delta \tilde B_s)H(s,y) \, \mu(ds,dy)\bigg].
    \end{align*}
    Now, Equation (III.3.16) and Theorem III.3.17 in \cite{jacod2013limit} allows us to identify $Y(t,y)=e^{\Gamma(t, y)}-\Delta \tilde B_t$, such that
   $\nu'(dt, dy) = Y(t,y) \nu(dt,dy) = (e^{\Gamma(t, y)}-\Delta \tilde B_t) \nu(dt, dy).$
\end{proof}

\begin{proposition}\label{rho(1)}
         Let $\Gamma=\Gamma(\omega;t,y)$ be an $\bbF^Y$-predictable function such that Equation \eqref{positivity condition} hold. The process $\rho(1)$ solves 
        \begin{align}
          \label{expression rho(1)}  
          d\rho_t(1) &=  \rho_{t-}(1)\int_{\R^n}  \bigg(\dfrac{1}{e^{\Gamma(t, y)}-\Delta \tilde B_t}-1\bigg) \, \tilde{m}(dt, dy)
        \end{align}
        where $\tilde{m}(dt, dy) = \mu(dt, dy) - \nu'(dt, dy)$. 
\end{proposition}
\begin{proof}
By Proposition \ref{thm: girsanov}, the Dol\'eans-Dade exponential $Z=\ccE \big((e^\Gamma-1)*\tilde \mu \big)$ is a strictly positive $(\bbF^Y,P)$-martingale whenever $E[Z_t] = 1$ for all $t \ge 0$.
Since $\Gamma$ is $\bbF^Y$-predictable, $Z$ is $\bbF^Y$-adapted and, hence $Z_t$ is $\ccF_t^Y$-measurable for each $t\ge 0$.
Therefore $\rho_t(1)=E_{P'}[Z_t^{-1}|\ccF_t^Y]=Z_t^{-1}$.

We compute the dynamics of $Z^{-1}$. By Equation \eqref{eq:dynZ}, 
$Z_t = Z_{t-}(e^{\Gamma(t, \Delta Y_t)}-\Delta \tilde B_t)$ and hence
$$\Delta Z_t^{-1} = Z_{t-}^{-1}\bigg(\dfrac{1}{e^{\Gamma(t, \Delta Y_t)}-\Delta \tilde B_t}-1\bigg).$$
Since $Z$ is a pure-jump process, 
\begin{align*}
    Z_{t}^{-1} = 1 + \int_0^t \int_{\R^n} Z_{s-}^{-1} \bigg(\dfrac{1}{e^{\Gamma(s, y)}-\Delta \tilde B_s}-1\bigg) \, \mu(ds,dy), \qquad t \ge 0.
\end{align*}
By definition of $\tilde m$,
\begin{align*}
    Z_{t}^{-1} &= 1 + \int_0^t Z_{s-}^{-1} \bigg[\int_{\R^n}  \bigg(\dfrac{1}{e^{\Gamma(s, y)}-\Delta \tilde B_s}-1\bigg) \tilde{m}(ds, dy)\\
    &\qquad \qquad \qquad\qquad+ \int_{\R^n}  \bigg(\dfrac{1}{e^{\Gamma(s, y)}-\Delta \tilde B_s}-1\bigg)\nu'(ds, dy)\bigg].
\end{align*}
We show that the last (predictable) integral vanishes. For $(\omega,t)$ such that $t \notin \{T_1(\omega), \ldots,T_K(\omega)\}$, we have $\nu(dt,dy)(\omega)=0$. This  implies $\nu'(dt,dy)(\omega)=0$. 
If, on the contrary, we have $t \in \{T_1(\omega),\dots,T_K(\omega)\}$, then 
\begin{align*}
     \int_{\R^n}  \bigg(\dfrac{1}{e^{\Gamma(t, y)}-\Delta \tilde B_t}-1\bigg)\nu'(dt, dy) &=  \int_{\R^n}  \big(1- (e^{\Gamma(t, y)}-\Delta \tilde B_t)\big)\nu(dt, dy)\\
     &=\sum_{i=1}^K \ind{t=T_i}\int_{\R^n}  \big(1- (e^{\Gamma(T_i, y)}-\Delta \tilde B_{T_i})\big)\nu(\{T_i\}, dy)\\
     &= \sum_{i=1}^K \ind{t=T_i}\bigg(\int_{\R^n} \nu(\{T_i\}, dy) - \int_{\R^n}(e^{\Gamma(T_i, y)}-\Delta \tilde B_{T_i})\big)\nu(\{T_i\}, dy)\bigg)\\
     &= 1-1=0,
\end{align*}
since $\displaystyle\int_{\R^n} \nu(\{T_i\}, dy)=1$ and, by the definition of $\Delta \tilde B_{T_i}$,
$$\int_{\R^n}(e^{\Gamma(T_i, y)}-\Delta \tilde B_{T_i})\big)\nu(\{T_i\}, dy)=\int_{\R^n}e^{\Gamma(T_i, y)}\,\nu(\{T_i\}, dy)-\Delta \tilde B_{T_i}=1.$$
Hence, $\rho(1)=Z^{-1}$ is a local martingale and the conclusion follows.
\end{proof}

Building on Proposition \ref{rho(1)} and the Kallianpur-Striebel formula, we now derive the Zakai equation governing the unnormalized conditional distribution  $\rho_t(\phi)$. To this end, 
we make the following assumption. Recall that by $F$ we denoted the distribution of $\eta_1$ and that the regular conditional distribution of $\Delta Y_{T_i}$, denoted by $F^i$, was introduced in Equation \eqref{conditional distribution of jump size:2}.
\begin{assumption}\label{existence of L^i}
    For each  $i=1,\dots,K$, there exists a $\ccF_{T_i-}^Y\otimes \ccB(\R^n)$-measurable, positive function $L^i: \Omega \times \R^n \to \R$ such that
\begin{align}\label{Eq: L^i}
    \int_A L^i(\omega,y) F^i(dy)(\omega) = \int_A F(dy), \qquad A \in \ccB(\R^n),
\end{align}
for  $P$-almost every $\omega \in \Omega$.
\end{assumption}
Given this assumption, we let
\begin{align}\label{eq:Gamma}
    \Gamma(\omega,T_i(\omega),y):= \log L^i(\omega,y), \quad i =1,\dots,K
\end{align}
and $\Gamma(\omega,t,y)=0$ whenever $t \not \in \{T_1(\omega),\dots,T_K(\omega)\}$. Moreover. we define
\begin{equation*}
    L := e^\Gamma,
\end{equation*}
such that $L(\omega,T_i(\omega),y)=L^i(\omega,y)$.
This choice of $\Gamma$ induces, by Equation \eqref{eq:dynZ} a change of measure, under which the observation becomes uninformative. This can already be seen in the following lemma.

\begin{lemma}\label{lem5.4}
Assume that Assumption \ref{existence of L^i} holds.
    Choosing $\Gamma$ as in \eqref{eq:Gamma}, implies that $\Delta \tilde B=0$. Additionally, $\nu'(\{T_i\},dy)=F(dy)$ for $i=1,\dots,K$.
\end{lemma}
\begin{proof}
    Recall from Proposition \ref{thm: girsanov} that $\tilde B=(e^\Gamma -1)*\nu$, such that for $t \not \in \{T_1(\omega),\dots,T_K(\omega)\}$, $\tilde B_t(\omega)=0$. Furthermore, by the definition of $\nu$ in Equation \eqref{equ2: expression of mu^p} together with Assumption \ref{existence of L^i},

     \begin{align*}
     \Delta \tilde B_{T_i} &= \int_{\R^n} (e^{\Gamma(T_i,y)}-1)F^i(dy)
     = \int_{\R^n} (L^i(y)-1)F^i(dy)=\int_{\R^n} F(dy)-1=0. \tag*{\qedhere}
 \end{align*}\qedhere
\end{proof}

For our second main result, the Zakai equation, we introduce the following notation. Recall the definition of $S(\phi)$ from Equation \eqref{def:S} and
let
\begin{align}
 \label{tilde U}    U(T_i,y)&:= \frac1{L^i(y)}-1, \quad i=1,\dots,K, \\
 \label{V(phi)}   V(\phi)(T_i,y)&:=\frac{S(\phi)(T_i,y)}{L^i(y)}, \qquad i=1,\dots,K \\
 \label{W}
 W(\phi)(T_i)&= \Delta (S(\phi)*\nu)_{T_i}
\end{align}
and $U(t,y)=V(\phi)(t,y)=0$ for all $t \not \in \{T_1,\dots,T_K\}.$ Furthermore, let
\begin{align}
    R(\phi)&:=\rho_-(1)\cdot V(\phi) - \rho_-(1) W(\phi) U + \big(\rho_-(\phi) +\rho_{-}(\ccA \phi)  \big)\cdot U.
\end{align}

\begin{theorem}[Zakai Equation] \label{Zakai equation}
    Assume that Assumption \ref{assumption 1} and Assumption \ref{existence of L^i} hold. Then, for  every  $\phi \in \mathcal{C}_b^{2}(\R^m; \R)$,
    \begin{align}\label{eq:zakai}
        \rho_t(\phi) &= \rho_0(\phi) +\int_0^t \rho_s(\ccL\phi)ds + \int_0^t \rho_{s-}(\ccA \phi)dA_s + M_t(\phi), \quad t\ge 0.
    \end{align}
    where $M(\phi)$ is the $(\bbF^Y,P')$-local martingale given by
    $$M(\phi)=\big(R(\phi)*\tilde m\big).$$
\end{theorem}

\begin{proof}[Proof of Theorem \ref{Zakai equation}]
First, under Assumption \ref{existence of L^i}, we may choose $\Gamma$ according to Equation \eqref{eq:Gamma}. Then, Lemma \ref{lem5.4} implies $\Delta \tilde B=0$ and with 
    Proposition \ref{rho(1)} we obtain the representation
    \begin{align} \label{eq:representation rho(1)}
        \rho(1)=( \rho_-(1) U * \tilde m).
    \end{align} 
    
    Second, by the Kallianpur-Striebel formula in  Equation \eqref{Striebel}, we have that $\rho_t(\phi)=\rho_t(1)\, \pi_t(\phi).$ 
    
    This allows us to apply 
     the product rule to this formula. Hence, by Equation \eqref{6} and the above representation of $\rho(1)$, we obtain that
    \begin{align}
        \nonumber \rho_t(1)\pi_t(\phi)&= \rho_0(\phi) + (\rho_-(1) \cdot \pi(\phi))_t + (\pi_-(\phi) \cdot \rho(1))_t + [\rho(1),\pi(\phi)]_t\\
       \nonumber &= \rho_0(\phi) + \int_0^t \rho_{s-}(1)\pi_{s-}(\ccL\phi)ds + \int_0^t\rho_{s-}(1)\pi_{s-}(\ccA \phi)dA_s + (\rho_-(1)\cdot S(\phi) * \tilde \mu)_t
       \\
        &+ (\pi_-(\phi)\cdot \rho_-(1)  U * \tilde m)_t + \sum_{i=1}^K \ind{T_i\le t} \Delta \rho_{T_i}(1) \, \Delta \pi_{T_i}(\phi). 
      \label{rho equation} 
    \end{align}
    In the last line we will  use once more $ \pi(\phi)\, \rho(1)=\rho(\phi)$ for the first term.

    As a next step, we consider the last term of Equation \eqref{rho equation}, which contains the joint jumps. Equation \eqref{6} and Proposition \ref{rho(1)} yield, as above,
    \begin{align}
        \Delta \rho_{T_i}(1) \, \Delta \pi_{T_i}(\phi) &=\Delta \rho_{T_i}(1) \cdot \pi_{T_i^-}\big(\ccA \phi \big) 
    + \Delta \rho_{T_i}(1)  \cdot \Delta (S(\phi)*\tilde \mu)_{T_i} .
    \label{eq:1231}
    \end{align}
    
    For the following, we  compute the second term on the left hand side of this equation. To this end, 
note that, by Equation \eqref{eq:representation rho(1)},  $\Delta \rho_{T_i}(1)=\Delta (\tilde U*\tilde m)_{T_i}$ with $\tilde U = \rho_-(1)  U$.
Next, recall that under Assumption \ref{existence of L^i}, $\nu' = L \nu$ with positive $L$, such that $\nu = L^{-1} \nu'$.
By the definition of $\tilde \mu$ in Equation \eqref{F Y-compensated random measure} and by the definition of $\tilde m$ in Proposition \ref{rho(1)}, we obtain that 
\begin{align}
\label{eq:tilde mu represented in tilde m}    
\tilde \mu = \tilde m + \nu'-\nu = \tilde m + (1-L^{-1}) \, \nu'. \end{align}
 Thus, the last term in Equation \eqref{eq:1231}
 equals    \begin{align}
      \notag  \Delta \rho_{T_i}(1) &\cdot \Delta (S(\phi) * \tilde \mu)_{T_i} =
         \Delta (\tilde U*\tilde m)_{T_i} \cdot \Delta (S(\phi)*\tilde \mu)_{T_i}  \\
        \notag &=
        \Big( \Delta (S(\phi)*\tilde m)_{T_i} +  \Delta (S(\phi)(1-L^{-1})*\nu')_{T_i} \Big) \cdot \Delta (\tilde U*\tilde m)_{T_i}
        \\
        &=\Delta (S(\phi)*\tilde m)_{T_i}\cdot \Delta (\tilde U*\tilde m)_{T_i} +  \Delta (S(\phi)(1-L^{-1})*\nu')_{T_i}  \cdot \Delta (\tilde U*\tilde m)_{T_i}. \label{cova}
    \end{align}
The first term can be computed as follows:
\begin{align}
    \Delta (S(\phi)*(\mu-\nu'))_{T_i}\cdot &\Delta (\tilde U*(\mu-\nu'))_{T_i} 
         = 
         \Delta (S(\phi)*\mu)_{T_i}\cdot \Delta (\tilde U*\mu)_{T_i}
         - \Delta (S(\phi)*\mu)_{T_i}\cdot \Delta (\tilde U*\nu')_{T_i}  \notag \\
         &-\Delta (S(\phi)*\nu')_{T_i}\cdot \Delta (\tilde U*\mu)_{T_i}
         +\Delta (S(\phi)*\nu')_{T_i}\cdot \Delta (\tilde U*\nu')_{T_i}.\label{eq:1258}
\end{align}
Using again that $\mu = \tilde m + \nu'$, this equals
\begin{align*}
\eqref{eq:1258}
         &= 
         \Delta (S(\phi) \cdot \tilde U * \tilde m)_{T_i} 
         + \Delta (S(\phi) \cdot \tilde U * \nu')_{T_i} \\
         &- \Delta (S(\phi)*\tilde m)_{T_i}\cdot \Delta (\tilde U*\nu')_{T_i}
         - \Delta (S(\phi)*\nu')_{T_i}\cdot \Delta (\tilde U*\nu')_{T_i}\\
         &-\Delta (S(\phi)*\nu')_{T_i}\cdot \Delta (\tilde U*\tilde m)_{T_i}
         -\Delta (S(\phi)*\nu')_{T_i}\cdot \Delta (\tilde U*\nu')_{T_i}
         +\Delta (S(\phi)*\nu')_{T_i}\cdot \Delta (\tilde U*\nu')_{T_i} \\
         &= \Delta (S(\phi) \cdot \tilde U * \tilde m)_{T_i} 
         + \Delta (S(\phi) \cdot \tilde U * \nu')_{T_i}-  \Delta (S(\phi)*\tilde m)_{T_i} \cdot \Delta(\tilde U*\nu')_{T_i}\\
         &  - \Delta (S(\phi)*\nu')_{T_i} \cdot \Delta(\tilde U*\nu')_{T_i}-\Delta (S(\phi)*\nu')_{T_i}\cdot \Delta (\tilde U*\tilde m)_{T_i}.
    \end{align*}
    Therefore, by adding the last term of Equation \eqref{cova} into the final development of Equation \eqref{eq:1258}, we obtain
    \begin{align}
        \Delta \rho_{T_i}(1) \cdot \Delta (S(\phi) * \tilde \mu)_{T_i} &= \Delta (S(\phi) \cdot \tilde U * \tilde m)_{T_i} 
         + \Delta (S(\phi) \cdot \tilde U * \nu')_{T_i}-  \Delta (S(\phi)*\tilde m)_{T_i} \cdot \Delta(\tilde U*\nu')_{T_i}\notag\\
         & - \Delta (S(\phi)*\nu')_{T_i} \cdot \Delta(\tilde U*\nu')_{T_i} - \Delta (S(\phi)L^{-1}*\nu')_{T_i}  \cdot \Delta (\tilde U*\tilde m)_{T_i}
         \label{eq:1281}
    \end{align}
    Moreover, by Equation \eqref{tilde U}  
    we obtain that    $$\Delta(\tilde U*\nu')_{T_i}=\rho_{T_{i}^{-}}(1)\int_{\R^n}\bigg(\dfrac{1}{L^i(y)}-1\bigg)\nu'(\{T_i\},dy)=\rho_{T_{i}^{-}}\Delta \tilde B_{T_i}=0, $$
    where the last equality follows by Lemma \ref{lem5.4}, such that the third and the fourth term in Equation \eqref{eq:1281} vanish. Regarding the final term in this equation, we observe that
\begin{align*}
    \Delta (S(\phi)L^{-1}*\nu')_{T_i}  \cdot \Delta (\tilde U*\tilde m)_{T_i} &= \Delta (S(\phi)*\nu)_{T_i}  \cdot \Delta (\tilde U*\tilde m)_{T_i}
    = \Delta (W(\phi) \tilde U * \tilde m)_{T_i},
\end{align*}
by Equation \eqref{W}.
Summarising, the above findings yields that
\begin{align}
   \label{rho covariation} 
   \sum_{i=1}^K \ind{T_i\le t}\Delta \pi_{T_i}(\phi) \Delta \rho_{T_i}(1) 
  \notag &= \sum_{i=1}^K \ind{T_i\le t}\Delta \big(\pi_{-}(\ccA \phi)  \cdot \rho(1)\big)_{T_i}
   + (S(\phi) \cdot \tilde U * \tilde m)_t 
    + (S(\phi) \cdot \tilde U * \nu')_t\\ 
    &- (W(\phi) \tilde U * \tilde m)_t.
\end{align}

Using Equation \eqref{eq:representation rho(1)} on the first term in \eqref{rho equation}, we obtain that 
\begin{align}
    \Delta \big(\pi_{-}(\ccA \phi)  \cdot \rho(1)\big) &=
    \Delta \big(\pi_{-}(\ccA \phi) \,  \rho_-(1) \,  U * \tilde m\big) = 
    \Delta \big(K(\phi)*\tilde m\big)
\end{align}with $K(\phi):=\rho_-(\ccA \phi) U$.  Using again Equations \eqref{tilde U} and \eqref{eq:tilde mu represented in tilde m}, 
\begin{align*}
    \big(\rho_-(1)\cdot S(\phi)*\tilde \mu \big)_t &= \big(\rho_-(1)\cdot S(\phi)*\tilde m \big)_t + \big(\rho_-(1)\cdot S(\phi)(1-L^{-1})*\nu' \big)_t\\ 
    &=\big(\rho_-(1)\cdot S(\phi)*\tilde m \big)_t - (\rho_-(1)S(\phi) \cdot U * \nu')_t.
\end{align*}

Substituting this and Equation \eqref{rho covariation} in \eqref{rho equation}, we obtain:
\begin{align*}
    \lefteqn{\rho_t(\phi)- \rho_0(\phi) -\int_0^t \rho_s(\ccL\phi)ds - \int_0^t \rho_{s-}(\ccA \phi)dA_s  }\qquad   \\
    &= \big(\rho_-(1)\cdot S(\phi)*\tilde m \big)_t - (\rho_-(1)S(\phi) \cdot U * \nu')_t + \big(\rho_-(\phi)\cdot  U * \tilde m)_t  + \sum_{i=1}^K \ind{T_i\le t}\Delta \big(K(\phi)*\tilde m\big)_{T_i}\\  
    &+ (\rho_-(1) S(\phi) \cdot U * \tilde m)_t +(\rho_-(1)S(\phi) \cdot U * \nu')_t - (\rho_-(1)W(\phi)U * \tilde m)_t\\
    &= \big(\rho_-(1)\cdot S(\phi)*\tilde m \big)_t +\big(\rho_-(\phi)\cdot  U * \tilde m)_t +  \big(K(\phi)*\tilde m\big)_t + (\rho_-(1) S(\phi) \cdot U * \tilde m)_t\\
    &- (\rho_-(1)W(\phi)U * \tilde m)_t\\
    &= \big((\rho_-(1) S(\phi) + \rho_-(1) S(\phi) U) * \tilde m\big)_t
    + \big((\rho_-(\phi)+\rho_-(\ccA \phi)) U * \tilde m\big)_t -(\rho_-(1)W(\phi)U * \tilde m)_t \\
    &= M(\phi)
\end{align*}
and the result is proven.
\end{proof}

To illustrate the practical utility of the change of measure approach, we present a classic example: the Kalman filter with observations at predictable times. This examples clearly demonstrates how the abstract measure transformation simplifies the filtering problem by decoupling the signal and observation under the new measure.
\begin{example}[Linear Gaussian filtering]
    Consider a one-dimensional linear system where the signal $X_t$ evolves as an  Ornstein-Uhlenbeck (OU) process where observations occur at deterministic times $0<T_1<T_2<\ldots<T_K$: 
    $$dX_t =-\lambda X_t dt + \sigma dB_t+\sum_{i=1}^K \xi_i\delta_{T_i}(dt),$$ 
    where the jump sizes $\xi_i,\ i=1,\dots,K$ are i.i.d.\ with distribution $\ccN(0,Q)$ and are independent of the Brownian motion $B$. The observation process records noisy linear measurements, such that
    $$Y_t = \sum_{i=1}^K \ind{T_i \le t}(AX_{T_{i}^{-}} +\eta_i),$$
    $(\eta_i) \stackrel{\text{i.i.d.}}\sim \ccN(0,R)$
    where $A\in \R^{n \times m}$ is the  observation gain, and the noise $(\eta_i)$ is independent of $B$ and $(\xi_i)$.
We now study the measure change induced by Assumption \ref{existence of L^i} and Equation \eqref{eq:Gamma}. First,  since $\eta_1 \sim \ccN(0,R)$, $ F=P(\eta_1 \in \cdot)$ is simply the $\ccN(0,R)$-distribution, which we will denote by $F=\ccN(0,R)$ in the following. The conditional distribution of $\Delta Y_{T_i}= AX_{T_i-}+\eta_i$ given $\ccF_{T_{i}^{-}}^Y$, was denoted by $F^i$.

Using the same argument as in Example \ref{sec:Kalman filter} with $m_{T_{i}^{-}}=\hat X_{T_{i}^{-}}$ and $S_i =A^2P_{T_{i}^{-}}+R$, we have $$F^i=\ccN(A \hat X_{T_{i}^{-}},A^2P_{T_{i}^{-}}+R).$$ 
Following Equation \eqref{Eq: L^i}, we compute $L^i=dF / dF^i$. A simple calculation gives that \begin{align*}
        \Gamma(T_i,y)=\log  L^i(y)=\dfrac{1}{2}\log\bigg(\dfrac{A^2P_{T_{i}^{-}}+R}{R}\bigg)-\dfrac{y^2}{2R}+\dfrac{(y-A\hat X_{T_{i}^{-}})^2}{2(A^2P_{T_{i}^{-}}+R)}.
    \end{align*}
    This is the explicit form of the log Radon-Nikodym derivative between the target distribution $\ccN(0,R)$ and the conditional observation distribution $\ccN(Am_{T_{i}^{-}},A^2P_{T_{i}^{-}}+R)$. 
By Proposition \ref{thm: girsanov}, the compensator of the observation under $P'$, denoted by $\nu'$ equals
$$\nu'(dt,dy)=\sum_{i=1}^K \delta_{T_i}(dt)F(dy),$$
such that $(\Delta Y_{T_i})$ are i.i.d. $\ccN(0,R)$ and are independent of the signal process $X$. 
The unnormalized filter $\rho_t(\phi)$ satisfies the linear Zakai equation. 
 \end{example}

\section*{Acknowledgement}
      The authors thank Wilfried Kuissi Kamdem for helpful discussions and for reading the earlier version and David Criens for helpful discussions.
      Félix B. Tambe-Ndonfack is funded the Deutsche Forschungsgemeinschaft (DFG, German Research Foundation) – Project-ID 499552394 – SFB 1597 Small Data.

\end{document}